\documentclass[12pt]{amsart}
\usepackage[all]{xy}
\usepackage[parfill]{parskip}

\usepackage{verbatim}
\usepackage{color}

\usepackage{amsmath, amscd, graphicx, latexsym, hyperref, times}
\usepackage{color,graphicx}
\usepackage[abs]{overpic}
\usepackage{tikz}
\usepackage{tikz-cd}
\usetikzlibrary{arrows, patterns}

\oddsidemargin .25in \theoremstyle{plain}
\newtheorem{theorem}{Theorem}

\newtheorem{lemma}{Lemma}
\newtheorem{proposition}{Proposition}

\numberwithin{equation}{section}
\numberwithin{lemma}{section}
\numberwithin{proposition}{section}
\numberwithin{corollary}{section}
\numberwithin{remark}{section}

\begin{document}

\title{Liouville equations on complete surfaces with nonnegative Gauss curvature}
\author{Xiaohan Cai}
\address{School of Mathematical Sciences, Shanghai Jiao Tong University}
\email{xiaohancai@sjtu.edu.cn}
\author{Mijia Lai}
\address{School of Mathematical Sciences,
Shanghai Jiao Tong University}
\email{laimijia@sjtu.edu.cn}

\thanks{Lai's research is supported by NSFC No. 12031012, No. 12171313 and the Institute of Modern Analysis-A Frontier Research Center of Shanghai. Cai's research is supported by NSFC No.12171313. }
\date{}

\begin{abstract}
We study finite total curvature solutions of the Liouville equation $\Delta u+e^{2u}=0$ on a complete surface $(M,g)$ with nonnegative Gauss curvature. It turns out that the asymptotic behavior of the solution separates two extremal cases: on the one end, if the solution decays not too fast, then $(M,g)$ must be isometric to the standard Euclidean plane; on the other end,  if $(M,g)$ is isometric to the flat cylinder $S^1\times \mathbb{R}$, then solutions must decay linearly and are completely classified. 
\end{abstract}
\maketitle
\section{Introduction}
In the seminal work ~\cite{CL1} of Chen and Li, they obtained the radial symmetry of the solution of 
\begin{align} \label{E1}
    \Delta u+e^{2u}=0 
  \end{align} on $\mathbb{R}^2$,  provided that $  \int_{\mathbb{R}^2} e^{2u} dx<\infty $.
Put the center of symmetry at origin and up to a rescaling, then
\[
u(x)=\ln \left(\frac{2}{1+|x|^2}\right).
\]
The geometric meaning of above equation is that the conformal metric $g=e^{2u}g_0$ has constant Gauss curvature $1$.  It is tempting to think that $g$ is isometric to the standard round sphere. It is indeed true as the solution is the pull back of the round metric via stereogrpahic projection. Nevertheless this line of reasoning is valid only if one establishes the precise asymptotic behavior of $u$ at $\infty$, so that the metric extends to a smooth metric on the sphere from $\mathbb{R}^2$. The reader is refereed to ~\cite{LiTa} for this line of reasoning, see also ~\cite{GL1} for more general equations.  

The assumption $\int_{\mathbb{R}^2} e^{2u} dx<\infty$ is natural since there are infinitely many solutions of (\ref{E1}) with $\int_{\mathbb{R}^2} e^{2u} dx=\infty$. One way to obtain such solution is to pull back the spherical metric via a univalent holomorphic map from $\mathbb{C}$ to $\overline{\mathbb{C}}$. Recently, there appeared some interesting studies on (\ref{E1}) subject to $\int_{\mathbb{R}^2} e^{2u} dx=\infty$. Eremenko-Gui-Li-Xu ~\cite{EGLX} give a complete classification of solutions of (\ref{E1}) which are bounded from above. We also refer to ~\cite{GL2},~\cite{BEL}, ~\cite{L} for some studies on (\ref{E1}) from geometric point of view. 

The story in higher dimension was accomplished even earlier. For $n\geq 3$, let $u$ be a positive solution of 
\begin{align} \label{E2}
\Delta u+ u^{\frac{n+2}{n-2}}=0. 
\end{align}
We refer it as the scalar curvature equation as the conformal metric $g=u^{\frac{4}{n-2}} g_0$ has positive constant scalar curvature. Gidas-Ni-Nirenberg~\cite{GNN} first proved the radial symmetry of the solution under the assumption $u(x)\sim O(|x|^{2-n})$ as $|x|\to \infty$. This can be viewed as an analytical proof of a famous result of Obata on classification of constant scalar curvature metrics which are conformal to an Einstein metric. In a remarkable paper~\cite{CGS}, Caffarelli-Gidas-Spruck established the radial symmetry of the solution without any assumption on the asymptotic behavior of $u$.  

The scalar curvature equation for conformal metrics has critical Sobolev power. In the subcritical case, 
\begin{align}  \label{E3}
\Delta u+u^{p}=0, \quad 1<p<\frac{n+2}{n-2}.
\end{align}
Gidas-Spruck~\cite{GS} showed any nonnegative solution must be trivial. In a recent paper~\cite{CM}, Catino-Monticelli carried out a systematic study of above mentioned equations (\ref{E1},\ref{E2}, \ref{E3}) on complete manifolds with nonnegative Ricci curvature. Among many results, one particular case is a full extension of Caffarelli-Gidas-Spruck's result in dimension three to complete manifolds with nonnegative Ricci curvature.

Inspired by Catino-Monticelli's work, we aim to study the Liouville equation (\ref{E1}) on complete surfaces with nonnegative Gauss curvature, we are able to connect the asymptotic behavior of the solution with the underlying manifold.

To be more precise, let $(M,g)$ be a complete surface with nonnegative Gauss curvature. We study the Liouville equation 
\begin{align}\label{EL}
\Delta_g u+e^{2u}=0 
\end{align}
on $M$. A solution is called to have \textit{finite total curvature} if $\int_{M} e^{2u} dg<\infty$. 

In view of the Cohn-Vossen splitting Theorem, a complete surface $(M,g)$ with nonnegative Gauss curvature is 
\begin{itemize}
    \item either isometric to the flat cylinder $S^1\times \mathbb{R}$ or the flat M\"{o}bius band (unorientable),
    \item or diffeomorphic to $(\mathbb{R}^2, g_0)$.
\end{itemize}
In the latter case, by Huber's theorem \cite{Hu}, $(M,g)$ is conformal to $(\mathbb{R}^2, g_0)$.

Without loss of generality, we assume from now on that $M$ is orientable. In the former case, we have the following classification of solutions to (\ref{EL}).
\begin{theorem} \label{T1}
Let $u$ be a  solution of (\ref{EL}) with finite total curvature on the flat cylinder $(S^1\times \mathbb{R}, g_{prod})$. 
Then there exists $\mu\in[0,\infty)$ and $\beta\in (-1,\infty)$, such that either $\beta$ is an integer or $\mu=0$, and up to a rescaling, we have
\begin{align*}
e^{2u(z)}=\frac{(2\beta+2)^2|z|^{2\beta+2}}{(|1+\mu z^{\beta+1}|^2+|z|^{2\beta+2})^2} \quad \text{on } (\mathbb{C}-\{0\}, \frac{1}{|z|^2}g_0).
\end{align*}

\end{theorem}

The classification result is in fact not new. Since the Gauss curvature for the flat cylinder is identically zero, (\ref{EL}) has a geometric meaning that the conformal metric $e^{2u} g_{prod}$ has Gauss curvature $1$. Note the flat cylinder is conformal to $(\mathbb{R}^2\setminus \{0\}, g_0)$, thus (\ref{EL}) can be translated to the Liouville equation on $\mathbb{R}^2\setminus \{0\}$. Then the theorem follows from a combination of results of  Chou-Wan ~\cite[Theorem 5]{CW} , Chen-Li~\cite[Theorem 3.1]{CL2} and Troyanov~\cite[Theorem II]{Tr}.

Our main theorem is the following  rigidity result.

\begin{theorem} \label{T2}
Let $u$ be a solution of (\ref{EL}) with finite total curvature on a complete surface $(M,g)$ with nonnegative Gauss curvature. Let $r(x)$ be the distance function on $M$ with respect to a fixed point. If $u(x)\geq -2 \ln r(x)+o(\ln r(x))$, for $r(x)$ large,  then $(M,g)$ must be isometric to $(\mathbb{R}^2, g_0)$. Moreover, $-2$ is optimal in the sense that there exists non flat $(M,g)$ which admits solutions verifying $u(x)\sim \gamma \ln r(x)$ for any $\gamma<-2. $
\end{theorem}

A similar result has been proved in ~\cite[Theorem 1.10]{CM}. Our contribution here has two-fold. On the one hand, our assumption on $u$ is weaker than that in ~\cite{CM} and our treatment emphasizes the analysis of asymptotic behavior of the solution which helps to identify the threshold where the rigidity occurs. On the other hand, by setting the stage on the complete surfaces with nonnegative Gauss curvature, we unite two works of Chen-Li ~\cite{CL1} and ~\cite{CL2}. 

As mentioned above, we focus on the asymptotic behavior of the solution. If $(M,g)$ is conformal to $(\mathbb{R}^2, g_0)$, we assume $g=e^{2f}g_0$, then (\ref{EL}) becomes 
\begin{align} \label{EP}
\Delta u+e^{2f} e^{2u}=0 \quad \text{on $\mathbb{R}^2$.}
\end{align}
This is the so-called the prescribing Gauss curvature equation on $\mathbb{R}^2$, which has been investigated intensively over past few decades. Under suitable decay assumption of $e^{2f}$ near infinity, Cheng-Lin ~\cite[Theorem 1.1]{ChL} showed that the solution $u$ of (\ref{EP}) has the following asymptotic behavior 
\[
\lim_{x\to \infty} \frac{u(x)}{\ln |x|} =-\frac{1}{2\pi} (\int_{\mathbb{R}^2} e^{2f}e^{2u} dx)
\]
if and only if $\int_{\mathbb{R}^2} e^{2f}e^{2u} dx<\infty$. However, a priori, there is no any decay control for $e^{2f}$. In fact, $f$ satisfies an equation of similar type
\[
\Delta f+ K_g e^{2f}=0,
\]
where $K_g$ is the Gauss curvature of $g$. The only information here is that $K_g\geq 0$. Nevertheless, using Arsove-Huber's result~\cite{AH}, there exists an $m\in[0,1]$ and an exceptional set $E$ which is thin at infinity such that 
\begin{align} \label{MQ}
\lim_{x\to \infty, x\notin E}\frac{u(x)}{\ln|x|}=\liminf_{x\to \infty} \frac{u(x)}{\ln|x|}=-m. 
\end{align}
Here the thinness of a set at infinity is a concept concerning the logarithmic capacity. For a conformal metric  $e^{2f}g_0$ on $\mathbb{R}^n (n\geq 3)$ with nonnegative Ricci curvature,  Ma-Qing~\cite{MQ} obtained a similar asymptotic behavior for the conformal factor $f$. 

While Cheng-Lin and Arsove-Huber's works are main analytical inspirations for our work, we also benefit from two interesting geometric ingredients: the first is Li-Tam's work~\cite{LT} on comparison between the intrinsic distance and the Euclidean distance on $(\mathbb{R}^2, e^{2f}g_0)$ (see Lemma \ref{Lem.HLT}) and the second is an isoperimetric inequality on surfaces with nonnegative Gauss curvature established recently by Brendle~\cite{B} (see Lemma \ref{lemma.Brendle}). 

We present proofs in the next section. The natural analog on the study of equation (\ref{E2}) on higher dimensional complete locally conformally flat manifolds with nonnegative Ricci curvature will appear in a future work. 

\textbf{Acknowledgement:} Both authors wish to thank Prof. Shiguang Ma for helpful discussions.

\section{Proof of Main theorem}\label{S2}
\begin{proof}[Proof of Theorem~\ref{T1}]
Note that the flat cylinder $\mathbb{S}^1\times \mathbb{R}$ is conformal to $(\mathbb{R}^2\setminus\{0\},g_0)$ since 
\begin{align*}
    dt^2+d\theta^2=\frac{1}{r^2}dr^2+d\theta^2=\frac{1}{r^2}g_0,
\end{align*}
by setting $t=\ln r$. Let $e^{2w(x)}=\frac{1}{|x|^2}e^{2u(x)}$, then equation $\Delta_g u+e^{2u}=0$ is equivalent to 
\begin{align}\label{eq.on pct plane}
    \begin{cases}
        \Delta w+e^{2w}=0\quad \mathrm{on }\  \mathbb{R}^2\setminus \{0\},\\
        \int_{\mathbb{R}^2}e^{2w(x)}dx<\infty.
    \end{cases}
\end{align}
Chou-Wan's complex analysis argument ~\cite[Theorem 5]{CW} shows that
\begin{align*}
    w(x)=\beta_1\ln |x|+O(1)\quad \mathrm{as}\ x\to 0,\  \text{for some}\  \beta_1>-1.
\end{align*}
Let $\Tilde{w}(x)=w(\frac{x}{|x|^2})-2\ln|x|$, it is easy to see that $\Tilde{w}$ satisfies
\begin{align*}
    \begin{cases}
        \Delta \Tilde{w}+e^{2\Tilde{w}}=0\quad \mathrm{on }\  \mathbb{R}^2\setminus \{0\},\\
        \int_{\mathbb{R}^2}e^{2\Tilde{w}(x)}dx<\infty.
    \end{cases}
\end{align*}
Apply Chou-Wan's asymptotic result ~\cite[Theorem 5]{CW} to $\Tilde{w}$ and trace back to $w$, we get
\begin{align*}
    w(x)=\beta_2\ln |x|+O(1)\quad \mathrm{as}\ x\to \infty,\  \text{for some}\  \ \beta_2<-1.
\end{align*}
Therefore, $w(x)$ is a solution of (\ref{eq.on pct plane}) with conical singularities at $x=0$ and $x=\infty$. Hence the classification result of Troyanov ~\cite[Theorem II]{Tr}  yields, there exists $\mu\in[0,\infty)$ and $\beta\in(-1,\infty)$ such that either $\beta$ is an integer or $\mu=0$, and  up to a rescaling, we have
\begin{align*}
e^{2w(z)}=\frac{(2\beta+2)^2|z|^{2\beta}}{(|1+\mu z^{\beta+1}|^2+|z|^{2\beta+2})^2} \quad \text{on }\, \mathbb{C}-\{0\}.
\end{align*}
Then the desired result follows since $e^{2u(z)}=|z|^2 e^{2w(z)}$.
Note if both cone angles are less than $2\pi$ ($\beta\in(-1,0)$), Chen-Li~\cite[Theorem 3.1]{CL2} also obtained such classification.
\end{proof}

In this following, we give the complete proof of Theorem~\ref{T2}.

First we exclude the case of flat cylinder in Theorem~\ref{T2}. Suppose $u$ is a finite total curvature solution of (\ref{EL}) on the flat cylinder, then Theorem \ref{T1} implies
\begin{align*}
    u(x)\sim -(\beta+1) r(x),\quad \text{for $r(x)$ large,}
\end{align*}
where $\beta>-1$ is a constant. This contradicts with the assumption that $u(x)\geq -2\ln r(x)+o(\ln r(x))$ for $r(x)$ large. In conclusion, $(M,g)$ can not be the flat cylinder and thus is conformal to $(\mathbb{R}^2,g_0)$ by Cohn-Vossen splitting theorem and Huber's theorem.

Now we write $g=e^{2f}g_0$, then the finite total curvature solution $u$ of (\ref{EL}) becomes
\begin{align}\label{Efinite}
\begin{cases}
    \Delta u+e^{2f} e^{2u}=0 \quad \text{on $\mathbb{R}^2$,}\\
    \int_{\mathbb{R}^2}e^{2f+2u}dx<\infty.
\end{cases}
\end{align}

To fix the notation, we consider the quantity
\begin{align}\label{alpha}
    \alpha:=-\frac{1}{2\pi}\int_{\mathbb{R}^2}e^{2f+2u}dx.
\end{align}
The strategy of our proof is as follows: using the asymptotic lower bound assumption of the solution $u$, we establish a lower bound of $\alpha$ by analyzing carefully the asymptotic upper bound of the solution to (\ref{Efinite}). On the other hand, with the help of Brendle's isoperimetric inequality, we prove that the reversed inequality still holds. Hence the equality is obtained and the rigidity part of the isoperimetric inequality brings the rigidity of the underlying manifold.

First, we aim at getting the lower bound of $\alpha$. It is tempting to obtain a pointwise upper bound of the solution $u$ to (\ref{Efinite}) im terms of $\alpha$ so that the lower bound assumption on $u$ could imply immediately the lower bound of $\alpha$.  However, due to the lack of a uniform asymptotic behavior of the conformal factor $f$, it's impossible to derive such a pointwise bound for $u$. Instead, we shall give an upper bound of the integral average of $u$ on small balls. The argument is based on that of ~\cite{ChL}.

\begin{lemma}\label{Lem.upper bound}
    Let $(M,g)=(\mathbb{R}^2,e^{2f}g_0)$ be a complete surface with nonnegative Gauss curvature. Assume $u\in C^2(\mathbb{R}^2)$ is a solution to (\ref{Efinite}). Then for any $\epsilon>0$, $\sigma>0$, there exists $R>0$ such that for $|x|\geq R$ and $\rho=|x|^{-\sigma}$, there holds
    \begin{align*}
        \frac{1}{\pi \rho^2}\int_{B_{\rho}(x)}u(y)dy\leq (\alpha+\epsilon)\ln |x| +C,
    \end{align*}
    where $\alpha$ is given by (\ref{alpha}) and $C$ is a constant depends on $\epsilon, \sigma, R$.
\end{lemma}
\begin{proof}
    Construct an auxiliary function 
    \begin{align*}
        v(x)=\frac{1}{2\pi}\int_{\mathbb{R}^2}\psi(y)\ln(\frac{|x-y|}{|y|})dy,
    \end{align*}
    where $\psi(y)=e^{2f(y)+2u(y)}$.

    The proof is mainly composed of three claims:
    \begin{enumerate}
        \item $v(x)\leq -\alpha \ln |x| +C$ for $|x|\geq 2$.
        \item $u+v$ is a constant.
        \item For any $\epsilon>0, \sigma>0$, there exists $R>0$ such that for $|x|\geq R$ and $\rho=|x|^{-\sigma}$, there holds
        \begin{align}\label{ineq.upper_u}
            u(x)\leq 
            (\alpha+\epsilon)\ln |x|
            +\frac{1}{2\pi}\int_{B_{\rho}(x)}\psi(y)\ln (\frac{|y|}{|x-y|}) dy+C,
        \end{align}
    \end{enumerate}

    Proof of claim (1): For fixed $x$ with $|x|\geq 2$,
    \begin{align*}
        2\pi v(x)&=\int_{T_1}\psi(y) \ln(\frac{|x-y|}{|y|})dy
        +\int_{T_2}\psi(y)\ln(\frac{|x-y|}{|y|})dy
        +\int_{T_3}\psi(y)\ln(\frac{|x-y|}{|y|}) dy\\
        &\overset{\text{def}}{=}I_1+I_2+I_3,
    \end{align*}
    where $T_1=\{y\!: |y|\leq 2\}, 
    T_2=\{y\!: |y-x|\leq \frac{|x|}{2}, |y|\!\geq 2\},
    T_3=\{y\!: |y-x|\geq \frac{|x|}{2}, |y|\geq 2\}$.

    Note that for $|x|\geq 2$ and $y\in T_1$, we have $\ln |x-y|\leq \ln (|x|+2)\leq \ln |x|+\ln 2$. Thus
    \begin{align*}
        I_1&=\int_{T_1}\psi(y)\ln |x-y|dy-\int_{T_1}\psi(y)\ln |y|dy\\
        &\leq \int_{T_1}\psi(y)(\ln |x|+\ln 2) dy-\int_{T_1}\psi(y)\ln|y| dy\\
        &=(\ln |x|)\int_{T_1}\psi(y) dy+C.
    \end{align*}

    Now for $y\in T_2$, we have $|x-y|\leq \frac{|x|}{2}\leq |y|$. Hence
    \begin{align*}
        I_2\leq 0.
    \end{align*}

    For $y\in T_3$ and $|x|\geq 2$, there holds $|x-y|\leq |x|+|y|\leq |x||y|$. Therefore
    \begin{align*}
        I_3\leq (\ln |x|)\int_{T_3}\psi(y) dy.
    \end{align*}

    We conclude that
    \begin{align*}
        2\pi v(x)=I_1+I_2+I_3\leq -2\pi \alpha \ln |x| +C.
    \end{align*}
    The proof of claim (1) is finished.

    Proof of claim (2): It is easy to see that $\Delta v=e^{2f+2u}$ and $u+v$ is a harmonic function on $\mathbb{R}^2$. Hence there exists an entire function $f(z)$ such that Re$f=2(u+v)$. Let $F(z)=e^{f(z)}$. Clearly, by claim $(1)$ we get
    \begin{align*} 
        |F(z)|=e^{2u+2v}\leq C|z|^{-2\alpha}e^{2u},
    \end{align*}
    for $|z|\geq 2$. Using the lower bound (\ref{MQ}) for the conformal factor $f$ ($e^{2f}\geq |z|^{-2m}$), we get that for some $R_0$ large enough,
    \begin{align*}
        \int_{|z|\geq R_0} |F(z)||z|^{2\alpha}|z|^{-2m} dx
        \leq 
        C\int_{|z|\geq R_0} e^{2u} e^{2f} dx<\infty.
    \end{align*}

    Let $M(\rho)=\max_{|z|=\rho}|F(z)|$, we shall show that $M(\rho)\leq C\rho^{2m-2\alpha}$ for $\rho\geq R_0+1$. In fact, assume $|z_0|=\rho$ and $M(\rho)=|F(z_0)|$. The mean value property implies
    \begin{align*}
        |F(z_0)|
        \leq \frac{1}{\pi}\int_{B_1(z_0)}|F(z)| dx
        \leq \frac{1}{\pi}\int_{\rho-1\leq |z|\leq \rho+1}|F(z)|dx.
    \end{align*}
    Hence we get
    \begin{align*}
        M(\rho)\rho^{2\alpha-2m}
        &\leq \frac{1}{\pi}\int_{\rho-1\leq |z|\leq \rho+1}|F(z)|\rho^{2\alpha-2m}dx\\
        &\leq \frac{1}{\pi}\int_{\rho-1\leq |z|\leq \rho+1}|F(z)| |z|^{2\alpha-2m}(\frac{\rho}{\rho+1})^{2\alpha-2m}dx\\
        &\leq \frac{2^{2m-2\alpha}}{\pi}\int_{|z|\geq \rho-1}|F(z)||z|^{2\alpha-2m}dx<\infty.
    \end{align*}

    Therefore, the order of the entire function $F(z)$ is
    \begin{align*}
        \lambda:=\limsup_{\rho\to \infty} \frac{\ln \ln M(\rho)}{\ln \rho}=0.
    \end{align*}
    By a theorem of Hadamard (see Theorem 8 of p. 209 in ~\cite{A}), we conclude that the genus of $F(z)$ is zero and $F(z)$ is a constant since $F$ has no zeros. The proof of claim (2) is completed.

    Proof of claim (3): For any $\epsilon >0,\sigma >0$, choose $R>0$ large enough such that
    \begin{align*}
        (\sigma+1)\int_{|y|\geq R}\psi(y)dy\leq \pi \epsilon,
    \end{align*}
    where $\psi(y)=e^{2f(y)+2u(y)}$.
    By claim (2), we have
    \begin{align*}
        2\pi u(x)
        &=C+\int_{\mathbb{R}^2}\psi(y) \ln (\frac{|y|}{|x-y|}) dy\\
        &=C+\int_{\tilde{T}_1}\psi(y) \ln (\frac{|y|}{|x-y|}) dy
        +\int_{\tilde{T}_2}\psi(y) \ln (\frac{|y|}{|x-y|}) dy
        +\int_{\tilde{T}_3}\psi(y) \ln (\frac{|y|}{|x-y|}) dy\\
        &\overset{\text{def}}{=}\tilde{I}_1+\tilde{I}_2+\tilde{I}_3,
    \end{align*}
    where $\tilde{T}_1=\{y\!: |y|\!\leq R\}, 
    \tilde{T}_2=\!\{y\!: |y-x|\leq \frac{|x|}{2}, |y|\!\geq R\},
    \tilde{T}_3=\!\{y\!: |y-x|\geq \frac{|x|}{2}, |y|\!\geq R\}$.

    Now for $|x|\geq \frac{R^2}{R-1}$ and $y\in \tilde{T}_1$, we have $\ln |x-y|\geq \ln(|x|-R)\geq \ln |x|-\ln R$. Thus
    \begin{align*}
        \tilde{I}_1&=\int_{\tilde{T}_1}\psi(y)\ln |y| dy
        -\int_{\tilde{T}_1}\psi(y)\ln |x-y| dy\\
        &\leq C-(\ln |x|)\int_{\tilde{T}_1}\psi(y) dy
        +(\ln R)\int_{\tilde{T}_1}\psi(y) dy\\
        &\leq -(\ln |x|)\int_{\tilde{T}_1}\psi(y) dy+C.
    \end{align*}

     To estimate $\tilde{I}_2$, let $\tilde{T}^{\sigma}=\{y: |y-x|\leq |x|^{-\sigma}, |y|\geq R\}$. Then we have
    \begin{align*}
        \tilde{I}_2&=\int_{\tilde{T}^{\sigma}}\psi(y) \ln(\frac{|y|}{|x-y|})dy
        +\int_{|x|^{-\sigma}\leq |y-x|\leq \frac{|x|}{2}, |y|\geq R}\psi(y) \ln(\frac{|y|}{|x-y|})dy\\
        &\leq \int_{\tilde{T}^{\sigma}}\psi(y) \ln(\frac{|y|}{|x-y|})dy
        +\int_{|y|\geq R} \psi(y) \ln(\frac{\frac{3}{2}|x|}{|x|^{-\sigma}}) dy\\
        &\leq \int_{|y-x|\leq |x|^{-\sigma}}\psi(y) \ln(\frac{|y|}{|x-y|})dy
        +(\sigma+1)\int_{|y|\geq R}\psi(y)dy 
        +C.
    \end{align*}

Now for $y\in \tilde{T}_3$, one easily gets $|y|\leq 4|x-y|$. Therefore
\begin{align*}
    \tilde{I}_3=\int_{\tilde{T}_3}\psi(y)\ln(\frac{|y|}{|x-y|})dy
    \leq
    (\ln 4)\int_{\tilde{T}_3}\psi(y) dy\leq C.
\end{align*}

In conclusion, there holds 
\begin{align*}
    2\pi u(x)&=\tilde{I}_1+\tilde{I}_2+\tilde{I}_3\\
    &\leq C-(\ln |x|)\int_{|y|\leq R}\psi(y)dy
    +\int_{|y-x|\leq |x|^{-\sigma}}\psi(y)\ln (\frac{|y|}{|x-y|})dy
    +(\sigma+1)\int_{|y|\geq R}\psi(y)dy\\
    &\leq C+2\pi(\alpha+\epsilon)\ln |x|
    +\int_{|y-x|\leq |x|^{-\sigma}}\psi(y)\ln (\frac{|y|}{|x-y|})dy,
\end{align*}
for $|x|\geq \frac{R^2}{R-1}$.
The proof of claim (3) is completed.

Finally, we give the upper bound of the integral average of $u$. By Green's formula, we get
\begin{align*}
    u(x)=\frac{1}{\pi\rho^2}\int_{B_{\rho}(x)}u(y) dy+\frac{1}{2\pi}\int_{B_{\rho}(x)}\psi(y) \ln (\frac{\rho}{|x-y|}) dy,
\end{align*}
for every $x\in \mathbb{R}^2$ and $\rho>0$. Combined with (\ref{ineq.upper_u}), we have for any $\epsilon>0, \sigma>0$, there exists $R>0$ such that for $|x|\geq R$ and $\rho=|x|^{-\sigma}$,
\begin{align}\label{ineq.upper_average}
    \frac{1}{\pi\rho^2}\int_{B_{\rho}(x)}u(y) dy
    \leq 
    (\alpha+\epsilon)\ln |x|
    +\frac{1}{2\pi}\int_{B_{\rho}(x)}\psi(y) \ln(\frac{|y|}{\rho}) dy+C.
\end{align}
Note that $\frac{|y|}{\rho}\leq \frac{|x|+\rho}{\rho}=|x|^{\sigma+1}+1\leq |x|^{\sigma+2}$ for $|x|$ large enough, the second term in the right hand side of (\ref{ineq.upper_average}) could be estimated as
\begin{align*}
    \frac{1}{2\pi}\int_{B_{\rho}(x)}\psi(y)\ln (\frac{|y|}{\rho}) dy
    \leq
    \frac{\sigma+2}{2\pi}(\ln |x|)\int_{|y|\geq R/2}\psi(y) dy\leq \epsilon \ln|x|,
\end{align*}
for $|x|\geq R$ provided $R$ is large enough.
Inserting this into (\ref{ineq.upper_average}), the proof of the lemma is completed.
\end{proof}

To derive the lower bound of $\alpha$, we need a lower bound of $u$ in terms of the Euclidean distance $\ln |x|$ rather than the intrinsic distance $\ln r(x)$ appeared in the hypotheses of Theorem \ref{T2}. Fortunately, the comparison of these two distances is established by Li-Tam \cite[Corollary 3.3]{LT}. Moreover, Hartman \cite[Theorem 7.1]{Ha} revealed the connection between this limit and the asymptotic volume ratio of the manifold. Their results are combined as follows.

\begin{lemma}[Hartman, Li-Tam]\label{Lem.HLT}
    Let $(\mathbb{R}^2,e^{2f}g_0)$ be a complete manifold with nonnegative Gauss curvature $K$. Then
    \begin{align*}
        \lim_{x\to \infty}\frac{\ln r(x)}{\ln |x|}=1-\frac{1}{2\pi}\int_{\mathbb{R}^2}K dg=\beta,
    \end{align*}
    where 
    \begin{align*}
        \beta:=\lim_{t\to \infty} \frac{Area(B(p,t))}{\pi t^2} \in [0,1]
    \end{align*}
    is the asymptotic volume ratio of the manifold $(\mathbb{R}^2,e^{2f}g_0)$.
\end{lemma}

Given this asymptotic behavior of $r(x)$, the a prior assumption on $u$ could be applied to obtain the lower bound of $\alpha$ in terms of the asymptotic volume ratio.

\begin{proposition}\label{prop.lower}
    Let $(\mathbb{R}^2,e^{2f}g_0)$ be a complete surface with nonnegative Gauss curvature. Let $u$ be a solution of (\ref{Efinite}). Assume
    \begin{align*}
        u(x)\geq -2\ln r(x)+o(\ln r(x)),
    \end{align*}
    for $r(x)$ large. Then
    \begin{align*}
        \alpha\geq -2\beta,
    \end{align*}
    where $\alpha$ is given by (\ref{alpha}) and $\beta$ is the asymptotic volume ratio of $(\mathbb{R}^2,e^{2f}g_0)$.
\end{proposition}
\begin{proof}
    By our assumption on $u$ and Lemma \ref{Lem.HLT}, we get for any $\epsilon>0$, there exits $R>0$ such that for $r(x)\geq R$,
    \begin{align*}
        u(x)\geq -2\ln r(x) +o(\ln r(x))\geq (-2 \beta-2\epsilon)\ln |x|+o(\ln |x|).
    \end{align*}

    While Lemma \ref{Lem.upper bound} yields that for any $\epsilon>0, \sigma>0$, there exists $R>0$ such that for $|x|\geq R$ and $\rho=|x|^{-\sigma}$,
    \begin{align*}
        \frac{1}{\pi \rho^2}\int_{B_{\rho}(x)}u(y)dy\leq (\alpha+\epsilon)\ln |x| +C,
    \end{align*}
    where $C$ is a constant depends on $\epsilon,\sigma, R$.

    We conclude that for any $\epsilon>0, \sigma>0$, there exists $R>0$ such that for $|x|\geq R$ and $\rho=|x|^{-\sigma}$,
    \begin{align*}
        (\alpha+\epsilon)\ln |x|+C
        &\geq 
        \frac{1}{\pi\rho^2}\int_{B_{\rho}(x)}u(y)dy\\
        &\geq 
        (-2\beta-2\epsilon)\frac{1}{\pi \rho^2}\int_{B_{\rho}(x)}\ln |y|dy +o(\ln |x|)\\
        &\geq (-2\beta-2\epsilon)(\ln |x|-\epsilon)+o(\ln |x|).
    \end{align*}
    Let $x\to \infty$, we get $\alpha+\epsilon\geq -2\beta-2\epsilon$. Since $\epsilon$ could be arbitrarily small, we get 
    \begin{align*}
        \alpha\geq -2\beta.
    \end{align*}
\end{proof}

We shall see that the reversed inequality also holds and thus the equality is obtained. For this, we need the isoperimetric inequality on nonnegatively curved surfaces established by Brendle \cite[Corollary 1.3]{B}, and it also helps to get the rigidity of the underlying manifold in our setting.
\begin{lemma}[Brendle]\label{lemma.Brendle}
    Let $(M^2,g)$ be  a complete noncompact manifold with nonnegative Gauss curvature. Let $D$ be a compact domain in $M$ with boundary $\partial D$. Then
    \begin{align*}
        L(\partial D)^2\geq 4\pi \beta\  A(D),
    \end{align*}
    where $L(\partial D)$ and $A(D)$ represent the length of $\partial D$ and the area of $D$, respectively, and $\beta$ is the asymptotic volume ratio of $(M,g)$. Moreover, the equality holds if and only if $(M,g)$ is isometric to Euclidean space and $D$ is a ball.
\end{lemma}
Now with the help of Lemma \ref{lemma.Brendle}, one could mimic the argument in \cite{CL1} to give the upper bound of $\alpha$.
\begin{proposition}\label{prop.upper}
    Let $(\mathbb{R}^2,e^{2f}g_0)$ be a complete surface with nonnegative Gauss curvature. Let $u$ be a solution of (\ref{Efinite}). Then 
    \begin{align*}
        \alpha\leq -2\beta,
    \end{align*}
    where $\alpha$ is given by (\ref{alpha}) and $\beta$ is the asymptotic volume ratio of $(\mathbb{R}^2,e^{2f}g_0)$.
\end{proposition}
\begin{proof}
    Consider $F(t):=\int_{\Omega_t}e^{2u}dg$, where $\Omega_t=\{x:u(x)>t\}$ is the upper level set of $u$.
    
    The finite total curvature assumption $\int_M e^{2u}dg<\infty$ and the Minkowski inequality yield
    \begin{align*}
        A(\Omega_t)<\infty,
    \end{align*}
 where $A(\Omega_t)$ represents the area of $\Omega_t$ in $(\mathbb{R}^2,g=e^{2f}g_0)$.

    It follows from the equation (\ref{EL}) and the divergence theorem that
    \begin{align*}
        F(t)=\int_{\Omega_t}e^{2u}dg=-\int_{\Omega_t}\Delta u dg=-\int_{\partial \Omega_t}<\nabla u,\eta>dS_g=\int_{\partial \Omega_t}|\nabla u| dS_g.
    \end{align*}
    By the coarea formula, there holds
    \begin{align*}
        F'(t)=-\int_{\partial \Omega_t}\frac{e^{2u}}{|\nabla u|}dS_g=-e^{2t}\int_{\partial \Omega_t}\frac{1}{|\nabla u|}dS_g.
    \end{align*}
    Then the H{\"o}lder inequality and the isoperimetric inequality (Lemma \ref{lemma.Brendle}) imply
    \begin{align}\label{ineq.ode}
    \begin{split}
        (F^2(t))'
        &=-2e^{2t}\int_{\partial \Omega_t}|\nabla u|dS_g   \int_{\partial \Omega_t}\frac{1}{|\nabla u|}dS_g\\
        &\leq -2e^{2t}L(\partial \Omega_t)^2\\
        &\leq -8\pi \beta\ e^{2t}A(\Omega_t).
    \end{split}
    \end{align}
Note that the isoperimetric inequality still holds for noncompact regions whose area are finite, since the length of its boundary must be infinite by the completeness of $(\mathbb{R}^2,e^{2f}g_0)$.
    
    Finally integrating $(\ref{ineq.ode})$ from $-\infty$ to $\infty$ yields
    \begin{align*}
        -(\int_M e^{2u}dg)^2
        &\leq -8\pi \beta
        \int_{- \infty}^{\infty}
        e^{2t}A(\{x:e^{2u(x)}>e^{2t} \})dt\\
        &=-4\pi \beta\int_0^{\infty}A(\{x:e^{2u(x)}>\lambda\})d\lambda\\
        &=-4\pi \beta \int_Me^{2u}dg.
    \end{align*}
    Thus the desired inequality holds.
\end{proof}

\begin{proof}[Proof of Theorem \ref{T2}]
By Proposition \ref{prop.lower} and Proposition \ref{prop.upper}, we get
\begin{align*}
    \alpha=-2\beta.
\end{align*}
Inspecting the proof of Proposition \ref{prop.upper} shows that $L(\partial \Omega_t)^2=4\pi \beta\  A(\Omega_t)$ for every $t\in\mathbb{R}$. Hence Lemma \ref{lemma.Brendle} tells $(\mathbb{R}^2,e^{2f}g_0)$ must be isometric to the Euclidean space $(\mathbb{R}^2,g_0)$.

To see the sharpness of the coefficient $-2$ in the assumption $u(x)\geq -2\ln r(x)+o(\ln r(x))$, consider the following 
conformal flat manifolds:

Choose the conformal factor $
    e^{2f(x)}=\frac{\gamma}{(1+|x|^2)^{2-2\gamma}}.$
 Then for $\gamma \in [\frac{1}{2},1)$, $(\mathbb{R}^2,g=e^{2f}g_0)$ is a complete surface with nonnegative Gauss curvature $K_g=\frac{4(1-\gamma)}{\gamma(1+|x|^2)^{2\gamma}}$. 

Take $e^{2u(x)}=\frac{4}{(1+|x|^2)^{2\gamma}}$, it is easy to see that $\Delta u+e^{2f}e^{2u}=0$. In other words,
\begin{align*}
    \Delta_g u+e^{2u}=0.
\end{align*}
Moreover, $\int_{\mathbb{R}^2}e^{2u}dg=\int_{\mathbb{R}^2}\frac{4\gamma}{(1+|x|^2)^2} dx=4\pi\gamma<\infty.$

Direct computation shows
\begin{align*}
    &\lim_{x\to \infty}\frac{\ln r(x)}{\ln |x|}=2\gamma-1, \quad 
    \mathrm{for }\  \gamma\in
    (\frac{1}{2},1).\\
    &\lim_{x\to \infty}\frac{r(x)}{\ln |x|}=1,\quad 
    \mathrm{for }\  \gamma=\frac{1}{2}
\end{align*}
Thus for $\gamma\in(\frac{1}{2},1)$, we have
\begin{align*}
    u(x)\sim -2\gamma\ln |x| \sim -\frac{2\gamma}{2\gamma-1}\ln r(x),
\end{align*}
where $-\frac{2\gamma}{2\gamma-1}\in(-\infty,-2)$.

In conclusion, for any $k<-2$, there exists a complete surface $(\mathbb{R}^2,e^{2f}g_0)$ with nonnegative Gauss curvature which admits a finite total curvature solution $u$ of (\ref{EL}) with $u(x)\sim k\ln r(x)$. 

\end{proof}

We conclude this paper with following remark. The main theorem states the rigidity of the underlying manifold under the assumption $u(x)\geq -2 \ln r(x)+o(\ln r(x))$. However, on the other end, we cannot expect such rigidity as illustrated by examples above. More precisely, when $\gamma=\frac{1}{2}$, it readily follows that the solution $u$ decays linearly with respect to the distance induced by the metric. Hence one cannot distinguish the flat cylinder by linear decay condition on the solution.  

    Nevertheless, when the solution decays sufficiently fast, we can get the volume growth control of the underlying manifold. We record here as a result of independent interest. 

\begin{proposition} \label{P1}
    Let $(\mathbb{R}^2,e^{2f}g_0)$ be a complete surface with nonnegative Gauss curvature. Let $u$ be a solution of $(\ref{Efinite})$ satisfying
    \begin{align*}
        \liminf_{x\to \infty}\frac{u(x)}{\ln r(x)}=-\infty.
    \end{align*}
    Then the asymptotic volume ratio of $(\mathbb{R}^2,e^{2f}g_0)$ is zero.
\end{proposition}
\begin{proof}
    Suppose the asymptotic volume ratio $\beta$ is positive.  According to the claims (1) and (2) in Lemma \ref{Lem.upper bound}, there holds 
    \begin{align*}
        u(x)\geq \alpha\ln |x|+C,\quad \mathrm{for }\ |x|\geq 2,
    \end{align*}
    where $\alpha=-\frac{1}{2\pi}\int_{\mathbb{R}^2}e^{2f}e^{2u}dx$.
    Combined with Lemma \ref{Lem.HLT} one get
    \begin{align*}
        \liminf_{x\to\infty}\frac{u(x)}{\ln r(x)}
        =\liminf_{x\to\infty}\frac{u(x)}{\ln |x|}\frac{\ln |x|}{\ln r(x)}
        \geq \frac{\alpha}{\beta}>-\infty.
    \end{align*}
    This contradicts to the hypothesis. Hence we get $\beta=0$.
\end{proof}

We also have a partial converse to Proposition~\ref{P1}.

\begin{proposition} \label{P2}
    Let $(\mathbb{R}^2,e^{2f}g_0)$ be a complete surface with nonnegative and bounded Gauss curvature. Suppose the asymptotic volume ratio $\beta=0$, then there exists a solution of (\ref{Efinite}) satisfying
    \begin{align*}
        \lim_{x\to \infty}\frac{u(x)}{\ln r(x)}=-\infty.
    \end{align*}
\end{proposition}

\begin{proof}
    Since $f$ satisfies 
    \[
    \Delta f +K e^{2f}=0,
    \] where $0\leq K\leq C $ by assumption. Based on a work of Taliaferro~\cite{Ta}, Bonini-Ma-Qing ~\cite[Lemma 4.2]{BMQ} showed that 
 \[
 e^{2f}\sim |x|^{-2(1-\beta)}=|x|^{-2} \quad \text{as $|x|\to \infty$}.
 \]
In view of the existence theorem of McOwen~\cite[Theorem 1]{Mc}, for any $\alpha\in(-2,0)$, there exists a solution $u$ of (\ref{Efinite}) verifying 
\[
u(x)\sim \alpha \ln |x|+O(1) \quad \text{at $\infty$}. 
\]
Since Lemma~\ref{Lem.HLT} still holds for $\beta=0$, the conclusion readily follows. 

\end{proof}

\end{document}